\documentclass[a4paper,12pt,reqno]{amsart}

\usepackage{enumerate}

\usepackage{tikz,pgf,amscd,enumerate,multicol}
\usepackage{amsmath,amssymb,tikz-cd,tkz-graph}

\usepackage{float}

\makeatletter
\@namedef{subjclassname@1991}{2020  Mathematics Subject Classification}
\makeatother

\long\def\comment#1{}

\usepackage[pagebackref,colorlinks]{hyperref}

\usepackage[utf8]{inputenc}
\usepackage{amstext}
\usepackage{amsfonts}
\usepackage{amsthm}
\usepackage{textcomp}
\usepackage{amssymb}
\usepackage{amscd}
\usepackage{amsmath}
\usepackage{xcolor}

\makeatletter
\newtheorem*{rep@theorem}{\rep@title}
\newcommand{\newreptheorem}[2]{%
\newenvironment{rep#1}[1]{%
 \def\rep@title{#2 \ref{##1}}%
 \begin{rep@theorem}}%
 {\end{rep@theorem}}}
\makeatother

\theoremstyle{plain}
\newtheorem{thm}{Theorem}[section]
\newtheorem{prop}[thm]{Proposition}
\newtheorem{cor}[thm]{Corollary}
\newtheorem{lemma}[thm]{Lemma}
\newtheorem{conj}[thm]{Conjecture}

\theoremstyle{definition}
\newtheorem{deff}[thm]{Definition}
\newtheorem{example}[thm]{Example}

\theoremstyle{remark}
\newtheorem{rmk}[thm]{\bf Remark}

\newcommand{\Z}{\mathbb{Z}}

\newcommand{\reg}{\textnormal{reg}}

\newcommand{\spn}{\operatorname{span}}

\newcommand{\gr}{\operatorname{gr}}
\newcommand{\K}{\mathsf k}

\newcommand{\udim}{\operatorname{udim}}

\usepackage{tikz}

\begin{document}

\title[The uniform dimension of a talented monoid of a graph]{The uniform dimension of a  monoid \\ with applications to graph algebras}

\author[L. G. Cordeiro]{Luiz Gustavo Cordeiro}
\address{Luiz Gustavo Cordeiro: Departamento de Engenharias da Mobilidade - UFSC - Joinville - SC, Brazil}
\email{luiz.cordeiro@ufsc.br}

\author[D. Gon\c{c}alves]{Daniel Gon\c{c}alves}
\address{Daniel Gon\c{c}alves: Departamento de Matem\'{a}tica - UFSC - Florian\'{o}polis - SC, Brazil}
\email{daemig@gmail.com}

\author[R. Hazrat]{Roozbeh Hazrat}

\address{Roozbeh Hazrat: 
Centre for Research in Mathematics and Data Science\\
Western Sydney University\\
Australia} \email{r.hazrat@westernsydney.edu.au}




\begin{abstract} 
We adapt Goldie's concept of uniform dimensions from module theory over rings to $\Gamma$-monoids. A $\Gamma$-monoid $M$ is said to have uniform dimension $n$ if $n$ is the largest number of pairwise incomparable nonzero $\Gamma$-order ideals contained in $M$.  

Specializing to the talented monoid of a graph, we show that the uniform dimension provides a rough measure of how the graph branches out.
Since for any order ideal $I$, its orthogonal ideal $I^\perp$ is the largest ideal incomparable to $I$, we study the notions of orthogonality and regularity, particularly when $I^{\perp\perp}=I$.
 We show that the freeness of the action of $\mathbb Z$ on the talented monoid of a graph is preserved under quotienting by a regular ideal. Furthermore, we determine the underlying hereditary and saturated sets that generate these ideals. These results unify recent studies on regular ideals of the corresponding Leavitt path algebras and graph $C^*$-algebras.

We conclude that for graphs $E$ and $F$, if there is a $\mathbb Z$-monoid isomorphism $T_E\cong T_F$, then there is a one-to-one correspondence between the regular ideals of  the associated Leavitt path algebras $L_\K(E)$ and $L_\K(F)$ (and similarly, $C^*(E)$ and $C^*(F)$).   Since the talented monoid $T_E$ is the positive cone of the graded Grothendieck group $K_0^{\gr}(L_\K (E))$, this provides further evidence supporting the Graded Classification Conjecture.
\end{abstract}

\maketitle

\section{Introduction}

Let \( M \) be a right \( R \)-module, where \( R \) is a unital ring. The Goldie uniform dimension of \( M \), denoted \( \udim(M) \), measures how extensively a large part of \( M \), i.e. an essential submodule of \( M \), can be decomposed into a direct sum of uniform submodules, which serve as the smallest building blocks. A fundamental theorem states that any two essential submodules of \( M \) decompose into the same number of uniform submodules, ensuring that \( \udim(M) \) is well-defined. When $M$ is a vector space, uniform submodules become one-dimensional subspaces, and the entire space can be decomposed using them. Thus, the uniform dimension of \( M \) coincides with the usual dimension of the vector space \cite{lam}.

In this note, we adapt the notion of uniform dimension to the setting of a $\Gamma$-monoid $M$, i.e., a commutative monoid $M$ with an action of an abelian group $\Gamma$ on $M$ via automorphisms.
 We show that such a dimension is well-defined for a $\Gamma$-monoid $M$ and show that it corresponds to the largest integer $n$ for which there exist $n$ pairwise incomparable $\Gamma$-order ideals in $M$ (see~\S \ref{sectwo}).

For a $\Gamma$-order ideal $I$ of a $\Gamma$-monoid $M$, its \emph{orthogonal} ideal $I^\perp$ is defined by 
\[I^\perp := \{a\in M \mid a \parallel I \} \cup \{0\},\]
where $\parallel$ stands for incomparability with respect to the algebraic order defined on $M$. Since $I^\perp$ is the largest ideal incomparable to $I$, we study the notion of orthogonality and \emph{regularity}, i.e., when $I^{\perp\perp}=I$.

We then specialize to the setting of the talented monoid of a graph, which is a $\mathbb{Z}$-monoid. For a directed graph $E$, the talented monoid $T_E$ is defined as the free commutative monoid generated by vertices indexed by $\mathbb{Z}$, subject to the relation that each vertex equals the sum of its adjacent vertices (see Definition~\ref{talentedmon}). The monoid $T_E$ carries a natural $\mathbb{Z}$-action and turns out to be the positive cone of the graded Grothendieck group $K_0^{\mathrm{gr}}(L_\K(E))$, where $L_\K(E)$ denotes the Leavitt path algebra over the field $\K$ associated with the graph $E$. A longstanding conjecture, known as the Graded Classification Conjecture, asserts that the graded Grothendieck group (equivalently, the talented monoid \(T_E\)) serves as a complete invariant for classifying Leavitt path algebras (see Conjecture~\ref{alergy1}, \cite[\S 7.3.4]{lpabook}, and \cite{willie}).

The talented monoid of a graph has quite a rich structure, but working with it is straightforward. 
It provides control over the monoid's elements (such as minimal elements, atoms, etc.), and consequently on the geometry of the graphs (the number of cycles, their lengths, etc.). 
 As an example to demonstrate the richness of the talented monoid, consider the following graph $E$: 
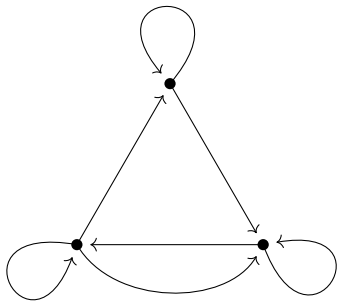
\begin{center}
\begin{tikzpicture}[scale=3.5]
\fill (0,0)  circle[radius=.6pt];
\fill (-.35,-.61)  circle[radius=.6pt];
\fill (.35,-.61)  circle[radius=.6pt];
\draw[->, shorten >=5pt] (0,0) to (.35,-.61);
\draw[->, shorten >=5pt] (.35,-.61) to (-.35,-.61);
\draw[->, shorten >=5pt] (-.35,-.61) to [in=-120,out=-60] (.35,-.61) ;
\draw[->, shorten >=5pt] (-.35,-.61) to (0,0);
\draw[->, shorten >=5pt] (0,0) to[in=130,out=50, loop] (0,0);
\draw[->, shorten >=5pt] (-.35,-.61) to[in=250,out=170, loop] (-.35,-.61);
\draw[->, shorten >=5pt] (.35,-.61) to[in=10,out=-70, loop] (-.35,-.61);
\end{tikzpicture}
\end{center}
It is easy to calculate the graph monoid (see Definition~\ref{def:graphmonoid}): $$M_E \cong \{0,x\}, \text{ with } x+x=x.$$ 
 However, using linear algebra results one can show that 
\begin{equation}
    T_E = \left\{ (a,b,c) \in \mathbb{Z}^3 \mid (a,b,c) \cdot \mathbf{z} > 0 \right\}\cup \{0\}, 
\end{equation}
\noindent where \( \mathbf{z} \) is the column vector  
\[
\mathbf{z} =
\begin{bmatrix}
    \frac{1}{3} \left( 1 - 5 \sqrt[3]{\frac{2}{11+3\sqrt{69}}} 
    + \sqrt[3]{\frac{1}{2} \left( 11+3\sqrt{69} \right)} \right) \\[10pt]
    \frac{1}{3} \left( -1 + \sqrt[3]{\frac{1}{2} \left( 25-3\sqrt{69} \right)} 
    + \sqrt[3]{\frac{1}{2} \left( 25+3\sqrt{69} \right)} \right) \\[10pt]
    1
\end{bmatrix},
\]

and the action of $\mathbb Z$ on $T_E$ takes the form  
$${}^1(a,b,c)=(a+c,a+b+c, b+c).$$

As mentioned earlier, we apply our general results on $\Gamma$-monoids to the talented monoid associated with a graph. More precisely, we study the uniform dimension of the talented monoid of a graph and show that it roughly captures how the graph branches out (see Example~\ref{uniformexam}). We also examine the orthogonality of  $\mathbb Z$-order ideals of the talented monoid and prove that  the freeness of the action of $\mathbb Z$ on the talented monoid of a graph  is preserved under quotienting by a regular ideal. We then determine the underlying hereditary and saturated sets that generate these ideals.  

We will relate our concepts of orthogonality and regularity in talented monoids of graphs to the corresponding notions in Leavitt path algebras. To clarify this connection, let $R$ be a ring and $I$ a two-sided ideal of $R$. Define the \emph{orthogonal ideal} to $I$ as 
$$I^\perp:=\{r\in R \mid rI=Ir=0\}.$$
Clearly, $I^\perp$ is an ideal and $I \subseteq I^{\perp\perp}$. An ideal $I$ is called \emph{regular} if $I= I^{\perp \perp}$. In the setting of graph algebras, i.e. Leavitt path algebras and graph $C^*$-algebras,  regular ideals have been the recent topic of investigations~\cite{galera, CCL, dd0, Vas1, Vas2}. It was shown that for Leavitt path algebras (or graph $C^*$-algebras), regular ideals are graded ideals (or gauge-invariant ideals) and the underlying generating sets of vertices associated with them have been determined.  To be precise, for a directed graph $E$ and a hereditary and saturated subset $H \subset E$, the \emph{orthogonal subset} $H^\perp$ is defined, and it is shown that $I(H)^\perp = I(H^\perp)$, where $I(H)$ is the ideal (or closed ideal) of the Leavitt path algebra $L_\K(E)$ (or graph $C^*$-algebra $C^*(E)$) generated by $H$. Consequently, it became possible to determine graphs where the ideals of the corresponding graph algebra are all regular.

In the setting of the talented monoid $T_E$ associated to a graph $E$, we show that $\langle H \rangle^\perp = \langle H^\perp \rangle$, where $\langle H \rangle$ is the order ideal of $T_E$ generated by $H$. We conclude that for row-finite graphs $E$ and $F$, if there is a $\mathbb Z$-monoid isomorphism $T_E\cong T_F$,  then the associated Leavitt path algebras (and graph $C^*$-algebras) have isomorphic lattices of regular ideals. This allows us to make a bridge between graph algebras and the graded $K$-theory invariant, as $T_E$ is the positive cone of the graded Grothendieck group $K^{\gr}_0(L_\K(E))$. 
 This provides more evidence for the Graded Classification Conjecture.

We focus on studying the regular ideals of the monoid \(T_E\) for row-finite graphs, particularly finite graphs, even though our approach can be extended to arbitrary graphs by suitably modifying the talented monoid \(T_E\) (see~\cite{liahaz}). One reason for this restriction is that the Graded Classification Conjecture~\ref{alergy1} is concerned with finite graphs. Additionally, the proofs and formulations are more transparent in the row-finite setting.

The paper is organized as follows: In Section~\ref{sectwo}, we define the notion of uniform dimension for a $\Gamma$-monoid. We study  the orthogonal ideal of a given $\Gamma$-order ideal of a $\Gamma$-monoid and consequently we define regular ideals in the setting of $\Gamma$-monoids. In Section~\ref{graphsdeffff}, we specialize to $\mathbb Z$-monoids arising from directed graphs, namely talented monoids. We determine the regular ideals of the talented monoid $T_E$ of a graph $E$ based on the underlying hereditary and saturated subsets that generate these ideals. 
 We show that the freeness of the action of $\mathbb Z$ on the talented monoid of a graph is preserved under quotienting by a regular ideal (Theorem~\ref{freeactiondom}). In Theorem~\ref{home20} we characterize talented monoids with uniform dimension $n$. We observe that  when $\udim(T_E)=n$, the graph $E$ branches out into $n$ strands which themselves do not branch out anymore.

In Section~\ref{secthree}, we relate the regular ideals of the talented monoid of the graph $E$ to those of the graph algebras associated to $E$. These results unify the recent studies on the regular ideals in both the corresponding Leavitt path algebras and graph $C^*$-algebras.

Finally, in Section~\ref{lastsecui} we completely determine the indecomposable $\mathbb Z$-order ideals of the talented monoid associated to a graph, and show how any order ideal can be expressed as a direct sum of such indecomposable ideals.  In Corollary~\ref{fgfyr67433}, we further unify the notion of decomposability of graph algebras and decomposability of their associated talented monoids.

\section{Ideals and annihilators in monoids}\label{sectwo}

Let $M$ be a commutative monoid written additively with unit \(0\). We define the \emph{algebraic pre-ordering} on the monoid $M$ by $a\leq b$ if $b=a+c$, for some $c\in M$.  We write $a<b$ if $b=a+c$ with $c\not=0$. Note that with this definition we can have $a<a$, for some $a\in M$. For \emph{nonzero} elements $a,b\in M$,  we write $a \parallel b$ if $a$ and $b$ are not comparable, i.e., $a \not =b$, $a\not < b$ and $b \not < a$. We further write $a\parallel S$ if $a$ is not comparable with any (nonzero) element of the subset $S$ of $M$.  We say the subsets $S_1$ and $S_2$ of $M$ are \emph{incomparable}, if $S_1\parallel S_2$, i.e., $a\parallel b$, for all nonzero elements $a\in S_1$ and $b\in S_2$.

A monoid is called \emph{conical} if $a+b=0$ implies $a=b=0$ and it is called \emph{cancellative} if $a+b=a+c$ implies $b=c$. An element $a\in M$ is called an \emph{atom} if $a=b+c$ implies $b=0$ or $c=0$. A nonzero element $a\in M$ is called \emph{minimal} if $b\leq a$ implies $a\leq b$.  When $M$ is conical and cancellative, the algebraic preordering is a partial order, so these notions coincide with the more intuitive definition of minimality, i.e., $a$ is minimal if $0\not = b\leq a$ implies $a=b$.

Throughout this note we mainly work with refinement monoids. Recall that a \emph{refinement} monoid is a commutative monoid, $M$, such that for any elements $a_0, a_1, b_0, b_1$ of $M$ such that $a_0+a_1=b_0+b_1$, there are elements $c_{00}, c_{01}, c_{10}, c_{11}$ of $M$ such that $a_0=c_{00}+c_{01}$, $a_1=c_{10}+c_{11}$, $b_0=c_{00}+c_{10}$, and $b_1=c_{01}+c_{11}$.

Consider now an abelian group $\Gamma$ acting on a commutative monoid $M$ by monoid automorphisms, i.e., $M$ is a \emph{$\Gamma$-monoid}.  For $\alpha \in \Gamma$ and $a\in M$, we denote the action of $\alpha$ on $a$ by ${}^\alpha a$. Note that, since the algebraic order is given in terms of the monoid operation, it is also preserved by the action of $\Gamma$. A monoid homomorphism $\phi:M_1 \rightarrow M_2$ is called $\Gamma$-monoid homomorphism if $\phi$ respects the action of $\Gamma$, i.e., $\phi({}^\alpha a)={}^\alpha \phi(a)$.


For a $\Gamma$-monoid $M$  we distinguish two types of submonoids. An \emph{order ideal} of $M$ is a submonoid $I$ which is also an ideal with respect to the natural order of $M$, i.e., if $x\in I$ and $y \leq x$, then $y\in I$. A \emph{$\Gamma$-order ideal} is an order ideal of $M$ which is closed under the action of $\Gamma$. Every $\Gamma$-order ideal $I$ of $M$ is a $\Gamma$-monoid on its own right, and the restriction of the natural order of $M$ to $I$ is the natural order of $I$.

Given $a\in M$, we denote by  $\langle a\rangle$ the $\Gamma$-order ideal generated by $a$. We have 
\begin{equation}\label{idealorderandnotorder}
\langle a\rangle = \big \{b\in M \mid b\leq\sum_{\alpha\in\Gamma} k_\alpha{}^\alpha a, \text{for some } k_\alpha\in\mathbb{N} \big \}.
\end{equation}

Next, we adapt the notion of \emph{essential ideals} and \emph{uniform ideals} in algebra to the setting of $\Gamma$-monoids and introduce the notion of uniform dimension for an arbitrary $\Gamma$-monoid. We will see that when the $\Gamma$-monoid is conical and refinement, our definition of uniform dimension coincides with the classical notion for rings (i.e., in terms of direct sums). For $\Gamma$-order ideals $I$ and $J$ of a $\Gamma$-monoid $M$, we write $I+J:=\{i+j\mid i\in I, j\in J\}$.

\begin{deff}\label{uniformdef}
Let \( M \) be a \(\Gamma\)-monoid. A \(\Gamma\)-order ideal \( I \) of \( M \) is called \emph{essential} if it has a nonzero intersection with every nonzero \(\Gamma\)-order ideal of \( M \). A nonzero \(\Gamma\)-order ideal \( I \) is said to be \emph{uniform} if any two nonzero \(\Gamma\)-order ideals contained in \( I \) have a nonzero intersection. Finally, a \(\Gamma\)-order ideal \( I \) is called \emph{superfluous} if for any \(\Gamma\)-order ideal \( J \), the equality \( I + J = M \) implies \( J = M \). 
\end{deff}

\begin{deff}\label{defuniformdem}
Let $M$ be a $\Gamma$-monoid. The \emph{uniform dimension} of $M$ is defined as 
\[
\udim(M)=\sup\{k \mid M \text{ contains $k$ pairwise incomparable nonzero  $\Gamma$-order ideals}\}.
\]

\end{deff}

Note that $\udim(M)=1$ if and only if $M$ is uniform.  

We write $I\subseteq_{e} M$ when $I$ is an essential  $\Gamma$-order ideal of $M$.  Note that the (essential) $\Gamma$-order ideal property is transitive: if $I$ is a (essential) $\Gamma$-order ideal of $M$ and $J$ is a (essential) $\Gamma$-order ideal of $I$, then $J$ is a (essential) $\Gamma$-order ideal of $M$. Note that by definition,  $I$ is a uniform ideal if and only if for any two nonzero ideals $J_1, J_2 \subseteq I$,  we have $J_1\cap J_2 \not = \{ 0\}$.   It follows that if the $\Gamma$-monoid $M$ is such that every nonzero $\Gamma$-order ideals is essential, then $M$ is uniform. Note also that the condition $I\cap J =\{0\}$ is equivalent to $I \parallel J$.

 For the next definition, recall the notion of comparability of elements, and that $\parallel$ is a binary relation used for nonzero elements of a monoid. 

\begin{deff}
    
    Let $M$ be a $\Gamma$-monoid and $I$ a $\Gamma$-order ideal of $M$. We define the \emph{orthogonal} set to $I$ as 
    \[I^\perp := \{a\in M \mid a \parallel I \} \cup \{0\}.\]
     We say that a $\Gamma$-order ideal $I$ is \emph{regular} if $I^{\perp\perp} = I$.  

\end{deff}

Note that $I\cap I^\perp =\{0\}$, i.e., $I\parallel I^\perp$.  
It is not hard to see $M\backslash I^\perp=\bigcup_{a\in M}(I+a) \backslash \{0\}$ which is a subsemigroup of $M$. Furthermore, if $\phi:M\rightarrow N$ is a $\Gamma$-isomorphism of $\Gamma$-monoids, then $\phi(I^\perp)=\phi(I)^\perp$.
Notice that $I \subseteq I^{\perp\perp}$.  
 Note also that if $M$ is refinement, $I\cap J=\{0\}$ if and only if any element of $I+J$ has a unique presentation, i.e., $i+j=i'+j'$ with $i,i'\in I$ and $j,j'\in J$ implies $i=i'$ and $j=j'$. Furthermore, if $M$ is refinement, $I+J$ is a $\Gamma$-order ideal. 
We write $I+J$ as the \emph{direct sum} $I\oplus J$ if  $I\cap J=\{0\}$. 

\begin{example}
    The sum of order ideals in a non-refinement monoid is not necessarily an order ideal. For example, let $M=\{t,b,l,r,m\}$ be the diamond semilattice, where the only order relations are that $b\leq x\leq t$ for all $x\in M$, considered as a commutative monoid under joins: $x+y=\sup\{x,y\}$. Then the order of $M$ as a monoid is the same as its original order as a lattice.
    
    In this case, the order ideals of $M$ are precisely the subsets $X$ of $M$ which are closed under joins and \emph{downsets} of $M$, i.e., those $X$ such that $x\in X$ and $y\leq x$ imply $y\in X$ (i.e., \emph{downsets}). Then, $I=\{l,b\}$ and $J=\{r,b\}$ are order ideals of $M$, but $I+J=M\setminus\{m\}$ is not.
\end{example}

For a monoid $M$, we denote $Z(M):=\{a\in M \mid a+b=0, \text{ for some } b\}$.  Clearly $Z(M)$ is a subgroup and if $I$ is an order ideal, then $Z(M)\subseteq I$. Since the talented monoids are conical, throughout we work with conical monoids. Note that $M$ is conical if and only if $Z(M)=\{0\}$. 



For the rest of the section, we assume that the monoids are conical and refinement. This allows us to form the $\Gamma$-order ideal $I+J$ from two $\Gamma$-order ideals $I$ and $J$. Furthermore, we can prove that $I^\perp$ is also a $\Gamma$-order ideal. Finally, we will show that under these assumptions, the notion of uniform dimension for monoids coincides with the classical definition via direct sums from the theory of rings and modules.

\begin{lemma}\label{lynchpinlem}
 Let $M$ be a conical refinement monoid, $I$ be a $\Gamma$-order ideal of $M$, and  $a_1, \dots, a_n\in M$, with $n \in \mathbb N$. Then, $\sum_{i=1}^n a_i \parallel I$ if and only if $a_i\parallel I$, for all $1\leq i \leq n$. 
\end{lemma}
\begin{proof}
We prove that $a+b \parallel I$ if and only if $a \parallel I$ and $b \parallel I$. The general case follows from induction. Throughout the proof we use the conicality of the monoid, i.e., $a+b\not= 0$ if and only if $a\not=0$ or $b\not = 0$. We proceed by contraposition.

Suppose that $a$ is comparable to some nonzero element $i$ of $I$. If $i\leq a$, then $i\leq a+b$ as well. If $a\leq i$ then $a\in I$ and $a\leq a+b$. In any case, $a+b$ becomes comparable with some nonzero element of $I$. By symmetry, the same holds if $b$ is comparable to some element of $I$.
Thus, if $a$ or $b$ are comparable to some nonzero element of $I$, so is $a+b$.

Now, suppose that $a+b$ is comparable to some nonzero element $i$ of $I$. If $a+b\leq i$ then $a+b\in I$ and, since $a,b\leq a+b$, we conclude that $a,b\in I$ as well. On the other hand, if $a+b\geq i$ then we may write $a+b=i+j$ for some $j$. By the refinement property, there are $c_{00},c_{01},c_{10},c_{11}$ such that
\[a=c_{00}+c_{01},\quad b=c_{10}+c_{11},\quad i=c_{00}+c_{10},\quad\text{and}\quad j=c_{01}+c_{11}.\]
Then, $a$ and $b$ are comparable to $c_{00}$ and $c_{10}$, respectively, both of which are elements of $I$. Since $i\neq 0$,  at least one of those elements must also be nonzero. Therefore, at least one of $a$ and $b$ is comparable to a nonzero element of $I$.
Thus, if $a+b$ is comparable to a nonzero element of $I$, then either $a$ or $b$ is also comparable to a nonzero element of $I$. \qedhere
\end{proof}

 For a conical refinement monoid, the orthogonal set of a $\Gamma$-order ideal is also a $\Gamma$-order ideal, as we prove below. 

\begin{prop}
    Let $M$ be a conical refinement $\Gamma$-monoid and $I$ a $\Gamma$-order ideal of $M$. Then, $I^\perp$ is a $\Gamma$-order ideal of \(M\).
\end{prop}
\begin{proof}

It is clear that $I^\perp$ is invariant by $\Gamma$. Let $a, b\in I^\perp$. If $a$ or $b$ is zero, then $a+b\in I^\perp$. If $a$ and $b$  are nonzero, then Lemma~\ref{lynchpinlem} implies $a+b\in I^\perp$. So $I^\perp$ is a submonoid of $M$. Suppose $a\in I^\perp$ and $b\leq a$. We need to show that $b\in I^\perp$. Suppose $b$ is not zero and $b\not \in I^\perp$. Then by definition, there is a nonzero $i\in I$ such that either $b\leq i$ or $i\leq b$. If $i\leq b$, then $b=i+j$ for some $j$. Since $a=b+b'$, it follows that $a=i+j+b'$, i.e., $a$ is comparable with a nonzero element of $I$, which is a contradiction. On the other hand, if $b\leq i$ then it follows that $b\in I$ which in turn implies that $a\in I^\perp$ is comparable with a nonzero element of $I$, again a contradiction. 
\end{proof}

Next, we prove that if a refinement monoid splits as the direct sum of an ideal and its annihilator, then the ideal must be regular.

\begin{prop}\label{orthoreg}
Let $I$ be a $\Gamma$-order ideal of the conical refinement $\Gamma$-monoid $M$. If $M=I\oplus I^\perp$ then $I$ is regular.
\end{prop}
\begin{proof}
    We have the inclusion $I \subseteq I^{\perp \perp}$. Now, suppose $a\in I^{\perp \perp}$. We will show that $a\in I$. If $a=0$ then we immediately have $a\in I$, so we may assume that $a\neq 0$, and thus $a \parallel I^\perp$. However, by the assumption, $a=x+y$, where $x\in I$ and $y\in I^\perp$. If $y\not =0$, then \(a\) would be comparable to a nonzero element of $I^\perp$, contradicting the condition that $a\in I^\perp$.
    Hence, $y=0$ and thus $a=x\in I$.
\end{proof}

\begin{rmk}
    The converse of Proposition~\ref{orthoreg} is not true. For example, let $X$ be a topological space and $M$ be the monoid whose elements are the open subsets of $X$, with operation given by union. Then, $M$ is a conical, non-cancellative, refinement monoid, and the natural order of $M$ is set inclusion.
    
    Given an open subset $U$ of $X$, define the order ideal
    $$I(U)=\left\{A\in M \mid A\subseteq U\right\}.$$
    Then, $I(U)^\perp$ consists of all $A\in M$ for which $A\cap B=\varnothing$, whenever $B$ is an open subset of $U$. This is equivalent to requiring that $A\subseteq\operatorname{int}(X\setminus U)$, so $I(U)^\perp=I(\operatorname{int}(X\setminus U))$.
    
    The map $U\mapsto I(U)$ is an injective order-preserving function, with inverse $I\mapsto \bigcup I$. In particular, $I(U)$ is regular if and only if
    $$U = \operatorname{int}(X\setminus\operatorname{int}(X\setminus U))=\operatorname{int}(\overline{U}),$$
    that is, $U$ is a regular open set.
    
    In the specific case of $X=[0,1]$ with its standard topology, take $U=[0,1/2)$. Then $I(U)$ is regular, but $I(U)\oplus I(U)^\perp$ consists of all open subsets of $X$ which do not contain $1/2$, and this is clearly not $M$.
\end{rmk}

\begin{prop}\label{finalbrushes}
Let $M$ be a conical refinement $\Gamma$-monoid, $S$  an index set, and  $\{ I_s \}_{s\in S}$ a set of nonzero $\Gamma$-order ideals of $M$. Then:

\begin{enumerate}[\upshape(1)]

\item $(\sum_{s\in S}I_s)^\perp = \bigcap_{s\in S}I^\perp_s$.

\item If $I$ is an essential ideal then $I^\perp=\{0\}$.

\item We have $\bigoplus_{s\in S}I_s$ if and only if $\{ I_s \}_{s\in S}$ are pairwise disjoint.
\end{enumerate}
    \end{prop}
\begin{proof}
\begin{enumerate}[\upshape(1)]
    \item Let $0\not = x\in (\sum_{s\in S}I_s)^\perp$. Thus $x \parallel \sum_{s\in S}I_s$. Since for any $s\in S$, $I_s\subseteq \sum_{s\in S}I_s$,  it follows that $x \parallel I_s$, for any $s\in S$. Thus $x\in \bigcap_{s\in S}I^\perp_s$. Conversely, suppose $0\not= x\in \bigcap_{s\in S}I^\perp_s$. If $x \not \in (\sum_{s\in S}I_s)^\perp $, then there is an $0\not= y\in \sum_{s\in S}I_s$ such that either $x\geq y$ or $x< y$. Write $y=y_1+\cdots +y_n$, where $0\not = y_i \in I_{s_i}$, $s_i\in S$, $1\leq i \leq n$. If $x\geq y$ then  $x\geq y \geq y_i$, $1\leq i\leq n$, which implies that $x \not \in I_{s_i}^\perp$, $1\leq i\leq n$, a contradiction. If $x<y=y_1+\cdots +y_n$, then $x+t=y_1+\cdots +y_n$, for some $t\in M$. Using the refinement property, there are $\{ a_{ij} \}_{1\leq i\leq 2, 1\leq j\leq n} \subseteq M$ such that 
    $$x=a_{11}+\cdots +a_{1n} \text{ and } t=a_{21}+\cdots +a_{2n} \text{ and } y_i=a_{1i}+a_{2i}, 1\leq i \leq n.$$
    Thus, on one hand, $x$ is comparable with $a_{1i}$, $1\leq i \leq n$. On the other hand these elements all belong to $I_i$ as $y_i\geq a_{1i}$, $1\leq i \leq n$. Since not all of them can be zero (as $x$ is not zero) it follows that $x \not \in I^\perp_i$, for some $1\leq i \leq n$. Again a contradiction. This completes the proof. 
    
\item For an $\Gamma$-order ideal $I$, we know that $I^\perp$ is also a $\Gamma$-order ideal and by definition $I\cap I^\perp =\{0\}$. If $I$ is an essential ideal, $I$ has to intersect any nonzero order ideal. This forces $I^\perp=\{0\}$. 

\item If we have a direct sum $\bigoplus_{s\in S}I_s$ then, by definition, 
$I_i \cap \sum_{i\not = s}I_s =
\{ 0 \}$. In particular, $I_i \cap I_j =\{0\}$ for any $i,j \in S$, where $i\not = j$. For the converse, let $x \in I_i \cap \sum_{i\not = s}I_s$, for some $i\in S$. Using the refinement property, along with an argument similar to part (1) and the fact that the ideals are pairwise disjoint, we conclude that \( x = 0 \).\qedhere
\end{enumerate}
\end{proof}





Analogous to the algebraic setting (see~\cite[Theorem~6.1]{lam}), we prove the following result, which enables us to define the uniform dimension as in the classical case of rings.

\begin{thm}\label{Steinitzlemma}
Let $M$ be a conical refinement $\Gamma$-monoid. Let $I=\bigoplus_{k=1}^m I_k$ and $J=\bigoplus_{l=1}^n J_l$ be essential $\Gamma$-order ideals of $M$, where all $I_k$'s and $J_l$'s are uniform $\Gamma$-order ideals of $M$. Then $m=n$.
 \end{thm}
\begin{proof}
Suppose $n\geq m$. We first show, step by step, that 
\begin{align}\label{stepbystep1}
& J_1\oplus I_2 \oplus I_3 \oplus \cdots \oplus I_m,\\ 
& J_1\oplus J_2 \oplus I_3 \oplus  \cdots \oplus I_m,\notag\\
&\cdots\notag\\
&J_1\oplus J_2 \oplus J_3\oplus \cdots \oplus J_m,  \notag  
\end{align}
are all essential $\Gamma$-order ideals of $M$. Since $J=J_1\oplus J_2 \oplus J_3\oplus \cdots \oplus J_n$ is also  essential, it follows that $n=m$. 

Set $\hat I := I_2\oplus \cdots \oplus I_m$. There is an $1\leq l\leq n$ such that $\hat I\cap J_l  = \{ 0\}$. Indeed, suppose that for all $1\leq l\leq n$,  $\hat I\cap J_l  \not = \{ 0\}$. Then, $\hat I\cap J_l \subseteq_{e} J_l$ as the $J_l$'s are uniform. Since the direct sum of essential ideals is essential, we obtain that
\[(\hat I\cap J_1) \oplus (\hat I\cap J_2) \oplus \cdots \oplus (\hat I\cap J_n) \subseteq_e J_1 \oplus J_2 \oplus \cdots \oplus J_n=J.\]
Thus, 
$\hat I \cap J \subseteq_e J \subseteq_e M$. 
Hence, \(\hat{I} \subseteq_e M\), which is a contradiction since \(\hat{I} \cap I_1 = \{0\}\).
Thus, there is a $l$ such that $\hat I \cap J_l= \{0\}.$ Relabeling if necessary, we suppose $\hat I \cap J_1= \{0\}.$ Set \[I'=J_1\oplus \hat I =J_1\oplus I_2 \oplus I_3 \oplus \cdots \oplus I_m.\]
We will show that $I'$ is essential. Observe that $I' \cap I_1 \not = \{0\}$. Otherwise, we have $I'\oplus I_1=I_1\oplus I_2 \oplus I_3 \oplus \cdots \oplus I_m \oplus J_1= I\oplus J_1$, which can't be as $I$ is essential and $J_1\not = \{0\}.$
Consider 
\begin{equation}\label{imposter}
( I' \cap I_1)\oplus I_2 \oplus \cdots \oplus I_m \subseteq_e  I_1\oplus I_2 \oplus \cdots \oplus I_m \subseteq_e M.
\end{equation} 
 Since the left hand-side of (\ref{imposter}) is contained in $I'$, it follows that $I' \subseteq_e M$. Thus, we have shown that \( I' = J_1 \oplus I_2 \oplus I_3 \oplus \cdots \oplus I_m \) is essential. By repeating this argument step by step along the list in (\ref{stepbystep1}), the theorem follows.
\end{proof}


Theorem~\ref{Steinitzlemma} ensures that if an essential \(\Gamma\)-order ideal \( I \) of \( M \) decomposes as a direct sum of \( n \) uniform ideals, then any decomposition of an essential ideal of \( M \) must also contain exactly \( n \) uniform ideals. This observation allows us to define the uniform dimension in a manner analogous to the classical definition for rings and modules (see~\cite[Definition~6.2]{lam}).

\begin{deff}
 Let \( M \) be a conical refinement \(\Gamma\)-monoid. We define \(\udim'(M) = n\) if there exists an essential \(\Gamma\)-order ideal \( I \) that decomposes as a direct sum of \( n \) uniform \(\Gamma\)-order ideals. If no such \( n \) exists, we set \(\udim'(M) = \infty\).
\end{deff}

The following statements can be proved mutatis mutandis as in the case of modules~\cite[Propositions~6.3 and 6.4]{lam} and so we leave the proofs out. 

\begin{lemma}\label{gfyd74dj}
Let $M$ be a conical refinement $\Gamma$-monoid. 

\begin{enumerate}[\upshape(1)]

\item If $\udim'(M)=n <\infty$, then any direct sum of nonzero $\Gamma$-order ideals $I=I_1\oplus \cdots \oplus I_k\subseteq M$ has $k\leq n$.

\item We have $\udim'(M)=\infty$ if and only if $M$ contains an infinite direct sum of nonzero $\Gamma$-order ideals.

\end{enumerate}
\end{lemma}

\begin{cor}
    Let $M$ be a conincal refinement $\Gamma$-monoid. Then $\udim'(M)=\udim(M)$.
\end{cor}
\begin{proof}
By Proposition~\ref{finalbrushes}(3), the definition of uniform dimension (Definition~\ref{defuniformdem}) takes the form
\[
\udim(M)=\sup\{k \mid M \text{ contains a direct sum of $k$ nonzero $\Gamma$-order ideals}\}.
\]
If $\udim(M)=\infty$, then by Lemma~\ref{gfyd74dj}(2), $\udim'(M)=\infty$. If  $\udim(M)<\infty$, again Lemma~\ref{gfyd74dj}(2) implies that $\udim'(M)<\infty$. Now, Lemma~\ref{gfyd74dj}(1) implies that 
 $\udim(M)= \udim'(M)$.
 \end{proof}

\section{The uniform dimension and regular ideals of a talented monoid}\label{graphsdeffff}

A graph $E := (E^0, E^1, r, s)$ consists of a countable set of vertices $E^0$, a countable set of edges $E^1$, and maps $r\colon E^1 \to E^0$ and $s:E^1 \to E^0$ identifying the range and source of each edge. A graph $E$ is \emph{row-finite} if $s^{-1}(v)$ is finite for each vertex $v$. Throughout this paper we only consider row-finite graphs. We say a vertex $v$ is \emph{regular} if 
$s^{-1}(v)\not= \varnothing$ (otherwise $v$ is called a \emph{sink}) and we denote the set of all regular vertices by $E^0_\reg$.
If $E$ is a graph, then a \emph{nontrivial path} is a sequence $\alpha := e_1 e_2 \ldots e_n$ of edges with $r(e_i) = s(e_{i+1})$ for $1 \leq i \leq n-1$.  We say the path $\alpha$ has \emph{length} $| \alpha| :=n$, and we let $E^n$ denote the set of paths of length $n$.  We consider the vertices in $E^0$ to be trivial paths of length zero.  We also let $E^{*} := \bigcup_{n=0}^\infty E^n$ denote the paths of finite length. 
The source and range maps of \(E^1\) are extended naturally to \(E^{*}\).
Let $H \subseteq E^0$. The set $H$ is \emph{hereditary} if whenever $e\in E^{1}$ satisfies
$s(e)\in H$, then $r(e)\in H$.
The set $H$ is \emph{saturated} if  for all $v\in E_\reg^0$, $r(s^{-1}(v))  \subseteq H$ implies $v\in H$. 

For a hereditary subset $H\subseteq E^0$, we say a path $\alpha=e_1e_2\cdots e_n$ \emph{flows} to $H$ if 
$s(e_n)\not \in H$ but $r(e_n)\in H$. The hereditary condition implies that for $1\leq i \leq n$, $s(e_i)\not \in H$. We recall (\cite[Lemma~2.0.7]{lpabook}) that for a hereditary subset $H$, setting  $H_0:=H$,  \emph{the hereditary and saturated closure}, denoted by $\overline H$ is the set 
\begin{equation}\label{herclos1}
    \overline H =\bigcup_{n\geq 1} H_n, \text{ where, }
    H_n = H_{n-1} \cup \{v\in E^0 \mid r(s^{-1}(v))\subseteq H_{n-1} \}.
\end{equation}

Given vertices $u$ and $v$, we write $u\geq v$ if there is a finite path starting at $u$ and ending at $v$, i.e., a path with source $u$ and range $v$. Let \(w\) be a vertex of the graph \(E\). The \emph{tree rooted at \(w\)} is defined as
\[T(w):=\big \{r (\alpha)\mid\alpha\in E^{*},s(\alpha)=w\big \}=\left\{v\in E^0\mid w\geq v\right\}.\]
For a subset $X \subset E^0$, we define 
\[T(X):=\bigcup_{w\in X} T(w).
\]

A path $\mathfrak{c}=e_1e_2\dots e_n$ is called a \emph{closed path based at $v$} if $v=s(\mathfrak{c})=r(\mathfrak{c})$. A \emph{cycle} in $E$ is a closed path $\mathfrak{c}=e_1e_2\dots e_n$ such that $s(e_i)\not = s(e_j)$ for all $i\not = j$.  An \emph{exit} of a cycle $\mathfrak{c}=e_1\cdots e_n$ consists of an edge $f$ such that $s(f)=s(e_i)$ for some $i$ but $f\neq e_i$. The vertices $s(e_1),\ldots,s(e_n)$ are called the \emph{vertices of $\mathfrak{c}$}, and the set of these vertices is denoted by $\mathfrak{c}^0$, that is, $\mathfrak{c}^0=\{s(e_1),\ldots,s(e_n)\}$.  If $f$ is an exit of the cycle $\mathfrak{c}$, then a \emph{return} of $f$ to $\mathfrak{c}$ is a path $\mu$ such that $s(\mu)=r(f)$ and $r(\mu)\in \mathfrak{c}^0$. 

One can define a pre-ordering  on the set of sinks  and cycles of a graph $E$.  We say that a cycle $\mathfrak{c}$ \emph{connects} to a sink $z$  if $z\in T(\mathfrak{c})$, i.e., if there is a path from $\mathfrak{c}$ to $z$.  Similarly, a cycle $\mathfrak{c}$ \emph{connects} to a cycle $\mathfrak{d}$, if $T(\mathfrak{c}) \cap \mathfrak{d}^0 \not = \varnothing$, i.e., there is a path from $\mathfrak{c}$ to $\mathfrak{d}$.  This defines a pre-ordering on the set of cycles and sinks. Furthermore, this pre-ordering is a partial order if and only if the cycles in $E$ are mutually disjoint.  With this pre-ordering, a cycle $\mathfrak{c}$ is maximal if no other cycle connects to $\mathfrak{c}$ (in particular, a maximal cycle is disjoint from all other cycles). A sink $z$ is maximal if there is no cycle $\mathfrak{c}$ which connects to $z$.

If a cycle $\mathfrak{c}$ does/does not have an exit, then we say $\mathfrak{c}$ is \emph{cycle with/without exit}.  A graph is said to satisfy \emph{Condition~(L)} if every cycle has an exit. An \emph{extreme cycle} is a cycle that admits an exit, and such that every finite path that exits from it admits a return to it. We say that a cycle is a \emph{cycle with no return exit} if the cycle has an exit, however no exit returns to the cycle. 

The following Lemma, which is interesting in its own right, will be used in Theorem~\ref{decomtal} to decompose the graph monoids. 

\begin{lemma}\label{gfhfhfh22}
    Let $E$ be a row-finite graph and $H$ a hereditary subset of $E$. Suppose $v\in E^0 \, \backslash\,  H$. Then the following are equivalent:

    \begin{enumerate}[\upshape(1)]
\item $v \in \overline H \, \backslash\, H$;

\item there is a nonzero but finite number of paths that flow from each $w \in T(v)\, \backslash \, H$ to $H$;


\end{enumerate}
\end{lemma}

\begin{proof}
    (1) $\rightarrow$ (2) Suppose $v\in  \overline H \, \backslash\, H$. We first show that there is a path from any $w \in T(v)\, \backslash \, H$ which flows to $H$. Since $w\in T(v)$, there is a path $\alpha$ with $s(\alpha)=v$ and $r(\alpha)=w$. Since $v\in \overline H$, it follows that $w\in \overline H$. By (\ref{herclos1}), there is an $n\in \mathbb N$ such that $w\in H_n$. Let $n$ be the smallest positive number such that $w\in H_n \, \backslash \, H_{n-1}$. We argue by induction. For $n=1$, i.e., $w\in H_1\, \backslash \, H$, we have $r(s^{-1}(w)) \subseteq H$. Thus, there exists an edge $e$ which flows from $w$ to $H$. Furthermore, if there is a path from $w$ which flows to $H$, it has to start with one of the edges $r(s^{-1}(w))$ which already flows to $H$, so the path has to be one of these edges. Thus, there is a nonzero but finite number of paths from each \( w \in H_1 \) that flow to \( H \).
    
    Let $w\in H_{n+1}\, \backslash \, H_n$ and suppose that the claim is valid for any value less or equal to $n$. By definition, $r(s^{-1}(w)) \subseteq H_{n}$. There is an edge $e$, with $w=s(e)$ and $r(e)\in H_n\backslash H_{n-1}$, otherwise $w\in H_{n}$ which is not the case. By the induction step, there is a path $\alpha$ with source $r(e)$ which flows to $H$. Thus $e\alpha$ flows from $w$ to $H$. On the other hand, any path starting from $w$ first lands in $H_{n}$ via the finite number of edges $s^{-1}(w)$ and from $H_n$ there are only finite number of paths flows from $r(s^{-1}(w))$ to $H$. Therefore, there is a finite number of paths from \( w \) that flow to \( H \). 

    (2) $\rightarrow$ (1) If \( w \in T(v) \setminus H \), then by assumption, there is a nonzero but finite number of paths \( \alpha \) with \( s(\alpha) = w \) that flow to \( H \). We prove that $w\in \overline H$, by induction on the number of paths flowing to $H$.  
    
    First, suppose that there is a unique path $\alpha=e_1\cdots e_n$, where $n\geq 1$, with $s(\alpha)=w$ which flows to $H$. By definition of flows, $r(e_n)\in H$ but $s(e_n)\not \in H$. We claim there is no other edge that emits from $s(e_n)$. For, if there was another edge $f$ emitting from $s(e_n)$, then either $r(f)\in H$, or $r(f) \in T(w) \, \backslash \, H$. Either way, by the assumption,  there will be another path from $w$  which flows to $H$. This contradicts that $\alpha$ is the only path which flows from $w$ to $H$. Thus, $e_n$ is a unique edge emitting from $s(e_n)$ with $r(e_n)\in H$.  It follows that $s(e_n) \in \overline H$. We now work backward along the paths $\alpha=e_1\dots e_{n-1} e_n$. A similar argument shows that there is a unique edge emitting from $s(e_{n-1})$, with $r(e_{n-1})\in \overline H$. So $s(e_{n-1}) \in \overline H$. Repeating this process $n$ times we obtain $w\in \overline H$.
    
Now, for the induction hypothesis, suppose that if there are at most \( n \) paths from \( w \in T(v) \setminus H \) flowing into \( H \) then  \( w \in \overline{H} \).  Next, suppose there are $n+1$ paths which flows from $w$ to $H$. Let $ \alpha =e_1 e_2 \dots e_k$ be a path from $w$ flowing to $H$. Let $1\leq m\leq k$ be a minimal natural number such that there is an edge $g\not = e_m$ with $s(g)=s(e_m)$ (i.e., there is a bifurcation in the graph). Since $r(g)\in T(v)$, if $r(g)\not \in H$, then there are at most $n$ paths from $r(g)$ which flows to $H$. By the induction assumption, it follows that $r(g)\in \overline H$. Set $z:= s(e_m)$. We just showed that $r(s^{-1}(z))\subseteq \overline H$. It follows that $z=s(e_m)\in \overline H$. Now, there is a unique path $e_1\cdots e_m$ from $w$ which flows to $\overline H$. Arguing as in the first part, backward along the paths, we obtain that $w\in \overline H$.  

Since clearly $v\in T(v)\, \backslash \, H$, we obtain that $v\in \overline H$, as desired.
\end{proof}

Given a row-finite graph $E$, we denote by $F_E$ the free commutative monoid generated by $E^0$. We recall the notion of a graph monoid~\cite{lpabook,Pardo} introduced by Ara--Moreno--Pardo in relation with the $K_0$-group of the Leavitt path algebra of $E$, and its ``time evolution model'', i.e., the talented monoid of $E$
which was first introduced in this form in~\cite{hazli} and further studied in~\cite{Luiz}.

\begin{deff}\label{def:graphmonoid}
    Let $E$ be a row-finite graph. The \emph{graph monoid} of $E$, denoted $M_E$, is the commutative 
monoid generated by $\{v \mid v\in E^0\}$, subject to
\[v=\sum_{e\in s^{-1}(v)}r(e)\]
for every $v\in E^0_\reg$.
\end{deff}

The relations defining $M_E$ can be described more concretely: First, define a relation $\to_1$ on $F_E\backslash \{0\}$ as follows: for $\sum_{i=1}^n v_i  \in F$, and a regular vertex $v_j\in E^0$, where $1\leq j \leq n$,  
\[\sum_{i=1}^n v_i \to_1 \sum_{i\not = j }^n v_i+  \sum_{e\in s^{-1}(v_j)}r(e).\]
Then $M_E$ is the quotient of $F_E$ by the congruence generated by $\to_1$.

Let $\to$ be the reflexive and  transitive closure of $\to_1$ on $F_E\backslash \{0\}$ and $\sim$ be the congruence on $F_E$ generated by the relation $\to_1$.  The following lemma (\cite[Lemmas 4.2 and 4.3]{Pardo}) is essential to the remainder of this paper, as it allows us to translate the relations in the definition of $M_E$ in terms of the simpler relation $\to$ in $F_E$.

\begin{lemma}\label{confuu}
    Let $E$ be a row-finite graph.
    \begin{enumerate}[\upshape(1)]
        \item
            \textup{(The Confluence Lemma)} If $a,b\in F_E\setminus\left\{0\right\}$, then $a=b$ in $M_E$ if and only if there exists $c\in F_E$ such that $a\to c$ and $b\to c$. \textup{(}Note that, in this case, $a=b=c$ in $M_E$.\textup{)}
            
        \item
            If $a=a_1+a_2$ and $a\to b$ in $F_E$, then there exist $b_1,b_2\in F_E$ such that $b=b_1+b_2$, and  $a_1\to b_1$ and $a_2\to b_2$.
    \end{enumerate}
\end{lemma}

We are in a position to define the notion of the talented monoid of a directed graph. 
 
\begin{deff}\label{talentedmon}
    Let $E$ be a row-finite directed graph. The \emph{talented monoid} of $E$, denoted $T_E$, is the commutative 
    monoid generated by $\{v(i) \mid v\in E^0, i\in \mathbb Z\}$, subject to
    \[v(i)=\sum_{e\in s^{-1}(v)}r(e)(i+1),\]
    for every $i \in \mathbb Z$ and every $v\in E^{0}_\reg$. The additive group $\mathbb{Z}$ of integers acts on $T_E$ via monoid automorphisms by shifting indices: For each $n,i\in\mathbb{Z}$ and $v\in E^0$, define ${}^n v(i)=v(i+n)$, which extends to an action of $\mathbb{Z}$ on $T_E$. Throughout the paper, we denote the elements $v(0)$ in $T_E$ by $v$. 
\end{deff}

The talented monoid of a graph can also be seen as a special case of a graph monoid, which we now describe. The \emph{covering graph} of $E$ is the graph $\overline{E}$ with vertex set $\overline{E}^0=E^0\times\mathbb{Z}$, and edge set $\overline{E}^1=E^1\times\mathbb{Z}$. The range and source maps are given as
\[s(e,i)=(s(e),i),\qquad r(e,i)=(r(e),i+1).\] Note that the graph monoid $M_{\overline E}$ has a natural $\mathbb Z$-action given by ${}^n (v,i)= (v,i+n)$. 

The correspondence
    \begin{align*}
        T_{E} &\longrightarrow M_{\overline{E}}\\
        v(i) &\longmapsto (v,i)
    \end{align*}
    induces a $\mathbb Z$-monoid isomorphism. This  allows us to use the Confluence Lemma~\ref{confuu} for the talented monoid $T_E$ by identifying it with $M_{\overline E}$.

Let $E=(E^0,E^1,r,s)$ be a graph and $H$ be a hereditary  subset of $E^0$. We define the graph associated to $H$ as the graph $E_H=(E^0_H,E^1_H,r,s)$, where $E^0_H=H$ and $E^1_H=\{e\in E^1\mid s(e)\in H\}$. Thus for a hereditary subset $H$, one can also consider the $\mathbb Z$-monoid $T_{E_H}$, which we also denote by $T_H$.

There is a bijection between the lattice of hereditary saturated subsets of $E$, denoted by $\mathcal L(E)$, and the lattice of  $\mathbb{Z}$-order ideals of $T_E$, denoted by $\mathcal L(T_E)$, given by 
\begin{align}\label{corpara}
\mathcal L(E) &\longrightarrow \mathcal L (T_E)\\
H &\longmapsto \langle H\rangle, \notag
    \end{align}
where $\langle H\rangle$ is the $\mathbb Z$-order ideal of $T_E$ generated by the set $H$. 

 Note that $\langle H\rangle$ is naturally isomorphic as a $\mathbb{Z}$-monoid to $T_{H}$. Namely, there is an injection $T_H\rightarrow T_E; v\mapsto v$, with the image $\langle H \rangle$. This serves as an indication that working with the talented monoid $T_E$ is considerably more manageable than working directly with the Leavitt path algebra $L_{\mathsf k}(E)$ (see Remarks~\ref{parastoo}).
 
  Throughout we also use the fact that, for a hereditary saturated subset $H$ of $E^0$, and the corresponding $\mathbb Z$-order ideal $I=\langle H \rangle \subseteq T_E$, we have a $\mathbb Z$-module isomorphism 
\[T_{E/H}\cong T_{E}/I.\] We refer the reader to \cite{hazli} for these facts and the background concepts.

We say a hereditary subset $H$ is \emph{cofinal} if every vertex of $E$ is connected to some element of $H$, i.e., $E=R(H)$, where
\begin{equation}\label{gfhfyr667}
    R(H)=\{u\in E^0\mid  u\geq v, \text{ for some } v\in H\}.
    \end{equation}
We say a hereditary and saturated subset $H$ is \emph{connected}, if any two hereditary and saturated subsets $K, L\subseteq H$ nontrivially intersect.

\begin{prop}\cite[Proposition 6.8]{Luiz}\label{prop:Hcofinaliffzorderidealessential}
    Let $E$ be a row-finite graph and $H\subseteq E^0$ a hereditary subset. Then the $\mathbb{Z}$-order ideal $\langle H\rangle$ generated by $H$ in $T_E$ is essential if and only if $H$ is cofinal in $E$.
\end{prop}

We are in a position to characterize talented monoids with uniform dimension $n$. As the theorem below shows, when \( \udim(T_E) = n \), the graph \( E \) branches into \( n \) strands, which do not branch further (see Example~\ref{uniformexam}).

\begin{thm}\label{home20}
Let $E$ be a finite graph and $T_E$  its talented monoid. Then, $\udim(T_E)=n$ if and only if  there are $n$ pairwise disjoint connected hereditary and saturated  subsets  $H_1, H_2, \dots, H_n$ such that $\bigcup_{i=1}^n H_i$ is cofinal.
\end{thm}
\begin{proof}
   Let  $H_1, H_2, \dots, H_n$ be  $n$ pairwise disjoint connected hereditary and saturated  subsets  of $E$. Set $I_i:=\langle H_i \rangle$, where $1\leq i \leq n$. The disjointness of $H_i$'s implies that $I_i\cap I_j =\{0\}$, for $i\not = j$. Thus by Proposition~\ref{finalbrushes}(3), the ideals form a direct sum  
   $\bigoplus_{i=1}^n I_i$. 
   
Since \( \bigoplus_{i=1}^n I_i = \langle \, \bigcup_{i=1}^n H_i\,  \rangle \) and \( \bigcup_{i=1}^n H_i \) is cofinal, this ideal is essential. Moreover,  since each \( H_i \) is connected, any two \(\Gamma\)-order ideals \( A, B \subseteq I_i \) must intersect nontrivially; otherwise, they would give rise to two hereditary and saturated subsets \( K, L \subseteq H_i \) with \( K \cap L = \varnothing \), which is a contradiction. Thus, each \( I_i \) is uniform. We conclude that \( \udim(T_E) = n \).

   The proof of the converse is similar. 
\end{proof}

A \(\Gamma\)-order ideal \( I \) of a $\Gamma$-monoid $M$ is called \emph{superfluous} if for any \(\Gamma\)-order ideal \( J \), the equality \( I + J = M \) implies \( J = M \).

\begin{example}\label{uniformexam}

Consider the graphs \( E \), \( F \), and \( G \) below. We have \( \udim(T_E) = 1 \) and \( \udim(T_F) = \udim(T_G) = 2 \).  
Now, consider the \(\mathbb{Z}\)-order ideal \( I = \langle v \rangle \) in each of these three monoids. In \( T_E \), the ideal \( I \) is both essential and superfluous, whereas in \( T_F \), \( I \) is superfluous but not essential. Moreover, in \( T_G \), \( I \) is neither superfluous nor essential.

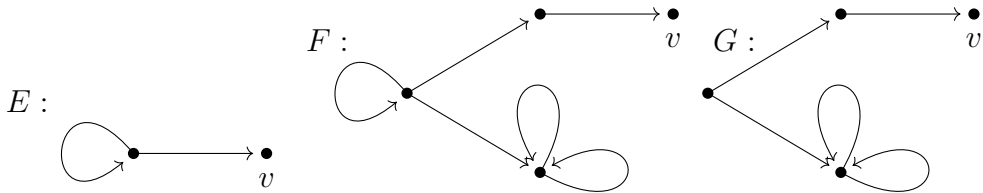
\begin{figure}[H]
\centering
\begin{tikzpicture}[scale=3.5]
\fill (0,0.2)  circle[radius=.6pt];
\draw (-0.4,.4) node{$E:$};
\fill (0.5,0.2)  circle[radius=.6pt];
\draw (0.5,0.1) node{$v$};
\draw[->, shorten >=5pt] (0,0.2) to (.5,0.2);
\draw[->, shorten >=5pt] (0,0.2) to[in=220,out=130, loop] (0,0.2);
\end{tikzpicture}
\begin{tikzpicture}[scale=3.5]
\fill (1,0)  circle[radius=.6pt];
\draw (0.7,.2) node{$F:$};
\fill (1.5,0.3)  circle[radius=.6pt];
\fill (1.5,-0.3)  circle[radius=.6pt];
\draw[->, shorten >=5pt] (1,0) to (1.5,0.3);
\draw[->, shorten >=5pt] (1,0) to (1.5,-0.3);
\draw[->, shorten >=5pt] (1.5,-0.3) to[in=120,out=60, loop] (1,0);
\draw[->, shorten >=5pt] (1.5,-0.3) to[in=30,out=-30, loop] (1,0);
\fill (2,0.3)  circle[radius=.6pt];
\draw[->, shorten >=5pt] (1.5,0.3) to (2,0.3);
\draw (2, 0.2) node{$v$};
\draw[->, shorten >=5pt] (1,0) to[in=220,out=130, loop] (1,0);
\end{tikzpicture}
\begin{tikzpicture}[scale=3.5]
\fill (2.1,0)  circle[radius=.6pt];
\draw (2.2,0.2) node{$G:$};
\fill (2.6,0.3)  circle[radius=.6pt];
\fill (2.6,-0.3)  circle[radius=.6pt];
\draw[->, shorten >=5pt] (2.1,0) to (2.6,0.3);
\draw[->, shorten >=5pt] (2.1,0) to (2.6,-0.3);
\draw[->, shorten >=5pt] (2.6,-0.3) to[in=120,out=60, loop] (2.5,-0.3);
\draw[->, shorten >=5pt] (2.6,-0.3) to[in=30,out=-30, loop] (2.5,-0.3);
\fill (3.1,0.3)  circle[radius=.6pt];
\draw[->, shorten >=5pt] (2.6,0.3) to (3.1,0.3);
\draw (3.1, 0.2) node{$v$};
\end{tikzpicture}

\end{figure}

\end{example}

Next, we prove a lemma that will be useful later. A variation of Lemma~\ref{pesto} with different proof can be found in \cite[Theorem 3.7(2)]{meteorgraphs}.

\begin{lemma}\label{pesto}
    Let $E$ be a finite graph and $T_E$ its associated talented monoid.  Then, for any $x\in T_E$ we have that
    $$x = \sum_{i=1}^n v_i(n_i),$$
    where $n_i\in \Z$ and the vertex $v_i$ is either a sink or belongs to $\mathfrak{c}^0$ for a cycle $\mathfrak{c}$.
\end{lemma}

\begin{proof}
    We proceed by induction on the size of $E^0$. If $|E^0|=1$ then the result is clearly true (either the only vertex is a sink or is the basis of a loop). 

    Assume that the result is valid for graphs such that $|E^0|$ is less than $k$ and let $E$ be a graph such that $|E^0|=k$. It is enough to show that any $w\in E^0$ can be written in the desired form. If $w$ already belongs to a cycle or is a sink, then there is nothing to do.

    Suppose that $w$ is not a sink nor belongs to $\mathfrak{c}^0$ for a cycle $\mathfrak{c}$. Let $V=r(s^{-1}(w))$ and $H$ be the hereditary closure of $V$. Since $w$ is not a sink, we have that $V$ is nonempty and, since $w$ does not belong to a cycle, we conclude that $H$ does not contain $w$. Thus, the subgraph $E_H$ of $E$ obtained by restricting the vertex set to $H$ is strictly smaller than $E$. Using the induction hypothesis, we find a collection of $v_i\in H$ (and integers $n_i$) such that each $v_i$ is a sink or belongs to a cycle in $E_H$, and for which
    $$\sum_{v\in V}v \to \sum_{i=1}^n v_i(n_i).$$
    Since $H$ is hereditary, all the $v_i$ which are sinks in $E_H$ are also sinks in $E$ and, clearly, any $v_i$ which belongs to a cycle in $E_H$ does so for $E$ as well. Moreover, since $H$ is hereditary, the relation $\sum_{v\in V}v \to \sum_{i=1}^n v_i(n_i)$ in $E_H$ may be realized in the same manner in $E$. So, we have
    $$w\to \sum_{v\in V}v(1)\to\sum_{i=1}^n v_i(n_i+1),$$
    with the $v_i$ being either a sink or belonging to $\mathfrak{c}^0$ for a cycle $\mathfrak{c}$ in $E$.\qedhere
\end{proof}

 Next, we recall a result that will be useful later, a proof of which can be found in \cite{hazli}.

\begin{thm}\label{brilllthem}
Let $E$ be a row-finite graph and $T_E$ its associated talented monoid. Then, $\mathbb Z$ acts freely on $T_E$ if and only if any cycle in $E$ has an exit, i.e., satisfies Condition~(L). 
\end{thm}

For an ideal generated by
a saturated hereditary subset of $E^0$, we will characterize its orthogonal ideal in $T_E$.  For this, we need the following definition. Recall the notation $R(H)$ from Equation~(\ref{gfhfyr667}). 

\begin{deff}\label{peropdef}
    Let $E$ be a directed graph and  $H$ a subset of \(E\). Define $H^\perp$ as the set of all vertices $v$ in $E$ for which there is no finite path starting at $v$ and ending at a vertex of $H$, that is, 
    \[H^\perp=E^0\setminus R(H).\]
     Moreover, we say that $H$ is \emph{regular} if $H^{\perp\perp} =H$.
\end{deff}

Note that if $H$ is hereditary then $H \subseteq H^{\perp\perp} $. Notice also that $H^\perp$ is a hereditary and saturated subset of $E^0$. 

Throughout the rest of the article, unless otherwise stated, $H$ stands for a hereditary and saturated subset of $E^0$, and so it corresponds to a $\Z$-order ideal $\langle H \rangle$ of the talented monoid $T_E$. 
    
    

\begin{prop}\label{marabeboos}
    Let $E$ be a row-finite graph and \(H\) a hereditary saturated subset of \(E^0\). Then, in the monoid \(T_E\), we have \(\langle H \rangle ^\perp=\langle H^\perp \rangle \).
\end{prop}
\begin{proof}
    Let \(v\in H^\perp\), and consider the associated element \(v\) of \(T_E\). We show that \(v\in \langle H \rangle ^\perp\), that is,  $v \parallel \langle H \rangle$. Since \(v\in H^\perp\), there is no finite path starting at $v$ that ends in $H$. Seeking a contradiction, suppose that $v$ is comparable with $\langle H \rangle$. Then there is an $x\in \langle H \rangle$ such that $v=x$, or $v > x$, or $v < x$. Note that $x$ can be written as a nontrivial sum of elements of the form \(w(i)\), where \(w\in H\) and \(i\in\mathbb{Z}\).
    
    Suppose 
     \(v >  x\) for some $x\in \langle H \rangle$. By Lemma \ref{confuu}, we may find elements \(z,t\) of the free monoid generated by \(E\times\mathbb{Z}\) in such a way that
    \[v\to z+t,\quad\text{ and } \quad x\to t.\]
    However, \(v\) belongs to \(H^\perp\), which is hereditary, and hence all the vertices that appear in \(z+t\) (and in particular those which appear in \(t\)) also belong to \(H^\perp\). Similarly, as \(x\to t\), the vertices that appear in \(t\) also belong to \(H\). Since \(x\) is nontrivial, so is \(t\). Therefore, picking any vertex that appears in the representation of $t$ as an element of the free monoid, we have found an element of \(H\cap H^\perp\), which is a contradiction as \(H\) and \(H^\perp\) are disjoint. Other cases can be argued similarly.
    
    We showed above that if \(v\in H^\perp\), then \(v\in \langle H \rangle^\perp\). As $\langle H \rangle^\perp$ is an order ideal, this implies that \(\langle H^\perp \rangle \subseteq \langle H \rangle^\perp\).
    
    Conversely, suppose that an element \(x\) of \(T_E\) does not belong to \(\langle H^\perp \rangle\). We need to prove that \(x\) also does not belong to \(\langle H \rangle^\perp\), that is \(x\geq y\) or $x<y$ for some \(y\in \langle H \rangle\setminus\left\{0\right\}\).
    
    Write \(x\) as a sum of elements of the form \(v(i)\), with \(v\in E^0\) and \(i\in\mathbb{Z}\). As \(x\not\in \langle H^\perp \rangle \), we have that some \(v\) in the representation of $x$  does not belong to \(H^\perp\).
    Since \(H^\perp=E^0\setminus R(H)\), we conclude that \(v\in R(H)\), that is, there is a path \(\alpha\) from \(v\) to some element of \(H\). In this manner, \(v(i)\geq y\), where \(y=r(\alpha)(i+|\alpha|)\) is a nontrivial element of \( \langle H \rangle \). Therefore,
    \[x\geq v(i)\geq y,\]
    and \(y\in \langle H \rangle \setminus\left\{0\right\}\), as desired.  \qedhere
\end{proof}

Note that, for a \(\mathbb{Z}\)-order ideal \(I\) of \(T_E\), the ideal \(I^{\perp\perp}\) is the smallest regular ideal of \(T_E\) that contains \(I\). Similarly, given a hereditary subset \(H\) of \(E^0\), where \(E\) is a row-finite graph, \(H^{\perp\perp}\) is the smallest regular hereditary saturated subset of \(E^0\) that contains \(H\).

Next, we give a geometric characterization of the ``regular closure'':

\begin{prop}\label{Avai}
    Let \(E\) be a graph. Then, for any  hereditary subset \(H\subseteq E^0\), we have that
    \begin{equation}\label{eq:doubleperpsubset}
    H^{\perp\perp}=\left\{w\in E^0\mid T(w)\subseteq R(H)\right\}.
    \end{equation}
    
    In other words, \(H^{\perp\perp}\) consists of all \(w\in E^0\) with the property that every path starting at \(w\) can be extended to a path ending at \(H\).
\end{prop}
\begin{proof}
    By definition,
    \[H^{\perp\perp}=E^0\setminus R(H^\perp)\quad \text{and}\quad H^{\perp}=E^0\setminus R(H).\]
    Then, for \(w\in E^0\), we have \(w\in H^{\perp\perp}\) if and only if there is no path from \(w\) to \(H^\perp\), which is to say that every path starting at \(w\) ends at \(R(H)\). This amounts to \(T(w)\subseteq R(H)\).\qedhere
\end{proof}

Note that the set on the right-hand side of \eqref{eq:doubleperpsubset} is always well-defined and hereditary (even for possibly non-hereditary $H$). As a consequence, we characterize regular subsets of $E^0$:

\begin{cor}\label{jambu}
    A subset \(H\subseteq E^0\) is regular  if and only if,
    \begin{equation}\label{hfgfyr7}
    H=\left\{w\in E^0\mid T(w)\subseteq R(H)\right\}.
    \end{equation}
\end{cor}

Recall from (\ref{corpara}) that there is a lattice isomorphism  $\phi\colon \mathcal L(E) \rightarrow \mathcal L(T_E); H \mapsto \langle H \rangle$. Proposition~\ref{marabeboos} now implies that, under $\phi$,  regular hereditary and saturated subsets are in one to one correspondence with regular $\mathbb Z$-order ideals.  

Next, we show that the regular ideals of the monoid preserve the freeness of the \(\mathbb{Z}\)-action under quotients. We refer the reader to \cite[Lemma~2.2]{hazli} and the preceding paragraph for the definition of a quotient monoid and the fact that talented monoids respect both graph and monoid quotients.

\begin{thm}\label{freeactiondom}
    Let $E$ be a row-finite graph, $T_E$ its associated talented monoid, and $I$ a regular $\mathbb Z$-order ideal of $T_E$.
        If $\mathbb Z$ acts freely on $T_E$, then $\mathbb Z$ acts freely on $T_E/I$ as well. 
\end{thm}

\begin{proof}
    Since $\mathbb Z$ acts freely on $T_E$ we have, by Theorem~\ref{brilllthem}, that all the cycles in $E$ have exits (i.e.,  $E$ satisfies Condition (L)). Let $I=\langle H \rangle$ for some hereditary and saturated subset $H$. Since $I$ is regular, $H$ is a regular set and thus it is described via the Equation~(\ref{hfgfyr7}) of Corollary~\ref{jambu}. Now, if $\mathbb Z$ does not act freely on $T_E/I\cong T_{E/H}$, then there is a cycle without exit in $E/H$.  This means that there is a cycle in $E \backslash H$ such that the ranges of all its exits are in $H$ (and because of Condition (L) in $E$, there is at least one exit). Using Equation~(\ref{hfgfyr7}) we obtain that the range of each edge that forms the cycle has to be in $H$, which is a contradiction. So, $\mathbb Z$ acts freely on $T_E/I$. 
\end{proof}

Next, we describe the graph conditions which guarantee that all the order ideals are regular. For our main theorem, we need the following lemma. 

\begin{lemma}\label{perghgyd}
    Let $X$ be a hereditary subset of $E$ which does not intersect the vertices of the cycle $\mathfrak{c}$, that is, such that $X\cap \mathfrak{c}^0=\varnothing$. Then, the hereditary saturated closure of $X$, $\overline X$, also does not intersect the vertices of $\mathfrak{c}$.
\end{lemma}
\begin{proof}
By hypothesis, $X$ is a hereditary subset of $ E^0\setminus \mathfrak{c}^0$. So, it is enough to prove that $E^0\setminus \mathfrak{c}^0$ is saturated, since in this case $\overline{X}\subseteq E^0\setminus \mathfrak{c}^0$ (see the description of $\overline{X}$ give in (\ref{herclos1})). To check that $E^0\setminus \mathfrak{c}^0$ is saturated, let $v\in E^0$ be such that $r(s^{-1}(v))\subseteq E^0\setminus \mathfrak{c}^0$. Then, $v\in E^0\setminus \mathfrak{c}^0$ as well, since otherwise we have $r(s^{-1}(v))\cap  \mathfrak{c}^0 \neq \varnothing$. 
\end{proof}

Finally, we characterize the finite graphs \( E \) in which all $\mathbb{Z}$-order ideals of \( T_E \) are regular.

\begin{thm}\label{allregular}
    Let $E$ be a finite graph. Then, every $\mathbb Z$-order ideal of the talented monoid $T_E$ is regular if and only if all cycles of $E$ are extreme or have no exits.
\end{thm}
\begin{proof}
    First, suppose that all cycles of $E$ are extreme or  have no exits. Let $H$ be a hereditary saturated subset of $E^0$. We show that $T_E=\langle H \rangle\oplus \langle H^\perp \rangle$. Let $x\in E^0$. By Lemma~\ref{pesto}, we can represent $x$ as
    $$x=\sum_{i=1}^N v_i(n_i),$$
    where each $v_i$ is either a sink or belongs to $\mathfrak{c}^0$ for a cycle $\mathfrak{c}$. Notice that it is enough to prove that if $v_i\notin H$ then  $v_i\in H^\perp$. So, suppose that $v_i\notin H$.
    
    If $v_i$ is a sink which does not belong to $H$ then there does not exist any (nontrivial) path starting at $v_i$. Hence there is no path that connects $v_i$ to $H$ and so, by Definition~\ref{peropdef}, $v_i\in H^\perp$.

    If $v_i$ belongs to $\mathfrak{c}^0$, for a cycle $\mathfrak{c}$, then no element of $\mathfrak{c}^0$  may belong to $H$, as $H$ is hereditary (since otherwise $v_i\in H$). We now consider two cases:
    
    Firstly, suppose that $\mathfrak{c}$ has no exits. Then any path starting at $v_i$ ends at $\mathfrak{c}^0$, so they do not end at $H$. Thus, $v_i\in H^\perp$. Secondly, assume that $\mathfrak{c}$ is extreme and there is a path connecting $v_i$ to a point of $H$. Then, the extreme property allows us to extend this path to a return to $\mathfrak{c}$. So, using again the hereditary property of $H$, we obtain that $v_i\in H$, which is not the case.  We conclude that there is no path from $v_i$ to $H$ and hence $v_i\in H^\perp$.


    Thus, we have shown that any $x\in E^0$ can be represented as a sum of elements of $\langle H \rangle$ and of $\langle H^\perp \rangle$ and therefore  $T_E=\langle H \rangle\oplus \langle H^\perp \rangle$. By Proposition~\ref{orthoreg}, $\langle H \rangle $ is regular.

    Conversely, suppose that $E$ has a cycle $\mathfrak{c}$ with a finite path that exits the cycle and does not have a return. Let $X$ be the set of vertices $v\in E^0$ such that there exists a path from $\mathfrak{c}^0$ to $v$, but there does not exist any path from $v$ back to $\mathfrak{c}^0$. Then $X$ is hereditary and does not intersect $\mathfrak{c}^0$. Let $H$ be the saturated closure of $X$. Then $H$ is hereditary and saturated and, by Lemma~\ref{perghgyd}, also does not intersect $\mathfrak{c}^0$. We prove that $H$ is not regular.

    Indeed, since $X$ is nonempty, there is a path $\alpha$ with $v:=s(\alpha)\in\mathfrak{c}^0$ and $r(\alpha)\in X\subseteq H$ (in particular $v\not\in H$). We check that $v\in H^{\perp\perp}$ using Proposition \ref{Avai}. Indeed, take any path $\beta$ starting at $v$. There are two possibilities:

    \begin{enumerate}
        \item There is no return from $r(\beta)$ to $\mathfrak{c}$. In this case, $r(\beta)\in X$.
        \item There is a return from $r(\beta)$ to $\mathfrak{c}$, that is, we may extend $\beta$ to a path $\widetilde{\beta}$ such that $r(\widetilde{\beta})\in\mathfrak{c}^0$. Extending $\widetilde{\beta}$ further along $\mathfrak{c}$, we may assume that $r(\widetilde{\beta})=v$, and extending it along with $\alpha$ we obtain an extension $\widetilde{\beta}\alpha$ of $\beta$ which ends in $X$.
    \end{enumerate}

    In any case, any path starting at $v$ may be extended to a path ending in $X$. Proposition \ref{Avai} now implies that $v\in H^{\perp\perp}$. But $v\not\in H$. Therefore $H^{\perp\perp}\neq H$ and hence $H$ is not regular.\qedhere
\end{proof}

\begin{rmk}
Theorem~\ref{allregular} can be extended to arbitrary graphs (with infinite vertices and infinite emitters) parallel to the case of Leavitt path algebras, as it is done in \cite[Theorem 3.14]{Vas1}. 
\end{rmk}

\section{Connection with regular ideals of graph algebras}\label{secthree}

 We refer the reader to the books \cite{lpabook, raeburn} for the basics of the theory of Leavitt path algebras and graph $C^*$-algebras. For a directed graph $E$, we denote the Leavitt path algebra associated to $E$ with coefficient from a field $\K$ by $L_\K(E)$ and the graph $C^*$-algebra by $C^*(E)$.

Let $A$ be a $\K$-algebra.
For a subset $X \subseteq A$, we define the \emph{orthogonal set} to (or annihilator of)  $X$ as 
$$ X^\perp := \{a \in A \mid ax= xa = 0, \text{ for all }x\in X\}. $$
Let $J$ be an ideal of $A$. Clearly 
$J \subseteq J^{\perp \perp}$.  We call an ideal $J\subseteq A$ a \emph{regular ideal} if $J=J^{\perp \perp}$.

\begin{rmk}
    We have now defined orthogonality for subsets of graphs, monoids, and algebras. We shall relate them in the context of graph algebras. There shall be no ambiguity as the specific context will clarify which notion of orthogonality is being referred to.
\end{rmk}

\begin{deff}
Let $E$ be a graph
and $I$ be an ideal in the Leavitt path algebra $L_\K(E)$ (or $C^*(E))$.
Define
\[ H(I) := \{ v \in E^0 \colon v \in I \}. \]
\end{deff}

If $H \subseteq E^0$ is hereditary, then we denote by $I(H)$ the ideal in $L_\K(E)$ (or $C^*(E)$) generated by $H$, that is, 
\begin{align}\label{jacare}
 I(H)  &:=\spn\{\gamma \lambda^*\mid \gamma, \lambda\in E^{*} \text{ and } r(\gamma)=r(\lambda)\in H\} \subseteq L_\K(E)\\
 I(H)  &:=\overline{\spn}\{\gamma \lambda^*\mid \gamma, \lambda\in E^{*} \text{ and } r(\gamma)=r(\lambda)\in H\} \subseteq C^*(E),\notag
\end{align}
see \cite[Lemma~2.4.1]{lpabook} (and \cite{raeburn}).

The following result gives a precise description of the graded/gauge-invariant ideals of a graph algebra in terms of subsets of the vertex set $E^0$.

\begin{thm} {\cite[Theorem 2.5.9]{lpabook}, \cite{raeburn}}\label{peixevoador} 
Let $E$ be a row-finite graph. Then the map $J\mapsto J\cap E^0$  is a lattice isomorphism between the graded ideals of $L_\K(E)$ \textup{(}or gauge-invariant ideal of $C^*(E)$\textup{)} and the hereditary saturated subsets of $E^0$, with inverse $H\mapsto I(H) $.

\end{thm}

\begin{prop}\label{bishap} Let $E$ be a row-finite graph and $J$ a graded ideal \textup{(}or gauge-invariant ideal\textup{)} of $L_\K(E)$ 
\textup{(}or $C^*(E)$\textup{)}.  Then, $H(J^\perp)=H(J)^\perp$.
\end{prop}
\begin{proof}
Set $J=I(H)$ for a hereditary and saturated subset $H$. By the proof of \cite[Proposition~3.5(i)]{dd0} (or \cite[Proposition~3.4(i)]{galera} for $C^*(E)$), we have that $H(J^\perp) =E^0\backslash R(H)= H(J)^\perp,$ where the last equality follows by definition.
\end{proof}

\begin{cor}\label{correspondence}
    Let $E$ be a row-finite graph. Let $T_E$ be the talented monoid, $L_\K(E)$ the Leavitt path algebra, and $C^*(E)$ the graph $C^*$-algebra associated with $E$, respectively. Then, there are one-to-one correspondences between the regular $\mathbb Z$-order ideals of $T_E$, the regular ideals of $L_\K(E)$, and the gauge invariant regular ideals of $C^*(E)$. 
\end{cor}
{\begin{proof}
Let $H$ be a hereditary and saturated subset of $E$. Let $I(H)$ be the graded ideal of $L_\K(E)$ generated by $H$. By Proposition~\ref{bishap}, 
$I(H)^\perp=I(H^\perp)$. It follows that $I(H)$ is regular if and only if $H^{\perp \perp} = H$. 
Proposition~\ref{marabeboos}, implies a similar result for the order ideal $\langle H \rangle$ of $T_E$. Namely,  $\langle H \rangle$ is regular if and only if $H^{\perp \perp} = H$.

Since there are bijections between $\mathcal L(E)$, the lattice of hereditary saturated subsets of $E$,  $\mathcal L(L_\K(E))$, the lattice of graded ideals of $L_\K(E)$, and those of the talented monoid $T_E$:
    \begin{align*}
    \mathcal{L}(L_\K(E)) &\longrightarrow \mathcal{L}(E) \longrightarrow \mathcal{L}(T_E) \\
    I(H) &\longmapsto H \longmapsto \langle H \rangle,
\end{align*}
    the algebraic part of the corollary follow immediately. The argument for the $C^*$-version is similar. 
\end{proof}}

By combining Theorem~\ref{allregular} with Corollary~\ref{correspondence}, we can directly obtain \cite[Theorem 3.14]{Vas2} for finite graphs from our analysis of regular ideals in talented monoids.

\begin{thm}\label{theorem_char_of_reflexive} 
Let $E$ be a finite graph. Then every ideal of the Leavitt path algebra $L_\K(E)$ \textup{(}or every  ideal of $C^*(E)$\textup{)}  is regular if and only if all cycles of $E$ are extreme or have no exits. 
\end{thm}

Recall that the Graded Classification Conjecture~\cite{willie} is as follows:

\begin{conj}\label{alergy1}
    Let \(E\) and \(F\) be finite graphs.    Then, the following are equivalent:
    \begin{enumerate}[\upshape(1)]

     \item The talented monoids \(T_E\) and \(T_F\) are \(\mathbb{Z}\)-isomorphic;

     \item The Leavitt path algebras \(L_\K(E)\) and \(L_\K(F)\) are graded Morita equivalent;
 
     \item The graph $C^*$-algebras \(C^*(E)\) and \(C^*(F)\) are equivariant Morita equivalent;
       
        \item The adjacency matrices  \(A_E\) and \(A_F\) are shift equivalent \textup{(}in the case of graphs with no sinks\textup{)}.
            
    \end{enumerate}
       \end{conj}

We can provide further evidence for the Graded Classification Conjecture.  

\begin{thm}
Let $E$ and $F$ be row-finite graphs. If the talented monoids \(T_E\) and \(T_F\) are \(\mathbb{Z}\)-isomorphic then there is one-to-one correspondence between regular ideals of $L_\K(E)$ and $L_\K(F)$ \textup{(}and similarly for $C^*(E)$ and $C^*(F)$\textup{)}. 
\end{thm}
\begin{proof}
First note that if $\theta:T_E\rightarrow T_F$ is a $\mathbb Z$-monoid isomorphism, then 
for any \(\mathbb{Z}\)-order ideal $I$ of $T_E$, 
$\theta(I)$ is an \(\mathbb{Z}\)-order ideal of $T_F$ and  $\theta(I)^\perp=\theta(I^\perp)$. This implies that there is a one-to-one correspondence between regular ideals of $T_E$ and $T_F$. The theorem now follows from Corollary~\ref{correspondence}. 
 \end{proof}

\begin{rmk}\label{parastoo}

The internal structure of a talented monoid $T_E$ seems to be  rather simpler than that of the associated Leavitt path algebra $L_\K(E)$. For example, for a hereditary and saturated subset $H\subseteq E$, the order ideal $\langle H \rangle \subseteq T_E$ is $\mathbb Z$-isomorphic to the talented monoid $T_H$, whereas for the Leavitt path algebra $L_\K(E)$, in order to realize the graded ideal $I(H)\subseteq L_\K(E)$ as a Leavitt path algebra, one needs to define a so-called porcupine graph~(see \cite{Vas3}).

\end{rmk}

\section{Geometric classification of indecomposable and simple order ideals of a talented monoid}\label{lastsecui}

Let $E$ be a finite graph. Recall the definitions of various types of cycles from \S\ref{graphsdeffff}. Similarly to what is done in \cite{Luiz}, let us denote
\begin{itemize}
    \item $P_s(E)$ the set of sinks of $E$.
    \item $P_{cc}(E)$ the set of vertices of $E$ which belong to cycles without exits.
    \item $P_{ec}(E)$ the set of vertices of $E$ which belong to extreme cycles.
    \item $P_{wc}(E)$ the set of vertices of $E$ which belong to cycles which have an exit without a return.
  \item $P_{mc}(E)$ the set of vertices of $E$ which belong to cycles with no return exit.
\end{itemize}

Clearly $P_{mc}(E) \subseteq P_{wc}(E)$. Let $P(E)$ be the union of the above sets, that is, the set of all sinks and vertices in the cycles of $E$. On $P(E)$, the binary relation
$\geq$ is a preorder. So, it induces an equivalence relation
$$v\sim w\iff v\geq w\text{ and }w\geq v$$
and a partial order on the quotient,  $P(E)/\!\sim$, as
\begin{equation}\label{pfjfhyrrrr}
[v]\geq[w]\iff v\geq w,    
\end{equation}
where $[v]$ denotes the $\sim$-equivalence class of $v$. 

Note that:
\begin{itemize}
    \item If $\mathfrak{c}$ is a cycle, then all the vertices of $\mathfrak{c}$ have the same $\sim$-equivalence class.
    \item If $v\sim w$, then either $v=w$, or there exists a cycle which contains both $v$ and $w$.
    \item $[v]$ is the intersection of the strongly connected component  of $v$ with $P(E)$.
    \item $P_s(E)$, $P_{cc}(E)$, $P_{ec}(E)$ and $P_{wc}(E)$ are $\sim$-invariant, in the sense that if $v$ belongs to any of these sets then the whole class $[v]$ is contained in that same set.
\end{itemize}
The first two observations, in particular, show that $P(E)/\!\sim$ can be thought as the ``collection of disjoint cycles of $E$''. We can thus use these elements to characterize the indecomposable ideals of the talented monoid $T_E$.

 Given any $\mathbb Z$-order ideal $I$ of $T_E$, let
    $$C(I)=\big \{[v] \mid v\in I \text{ and }  v\in P(E) \big \},$$ and $\max(C(I))$ be the set of maximal elements of $C(I)$ with respect to the ordering~(\ref{pfjfhyrrrr}).

\begin{thm}\label{lunhdfgdt}
    Given a finite graph $E$, and an element $[v] \in P(E)/\!\sim$, let $\langle v \rangle $ be the $\mathbb{Z}$-order ideal in $T_E$ generated by $v$. Then, 
   
    \begin{enumerate}[\upshape(1)]

        \item\label{it:independencepathconnectednessideals}
            $\langle v \rangle = \langle w \rangle $ if and only if  $[v]=[w]$ in $P(E)/\!\sim$.
        
        \item
            $\langle v \rangle $  is indecomposable as a $\mathbb{Z}$-order ideal.
        
        \item\label{it:conditionforsimplicity}
            $\langle v \rangle $  is simple as a $\mathbb{Z}$-order ideal if and only if  $v$ is a sink, or it belong to a cycle without exits or to an extreme cycle.
   
        \item\label{it:isod}
            For any $\mathbb Z$-order ideal $I$ of $T_E$, we have the decomposition $$I=\sum_{[v]\in \max(C(I))}\langle v \rangle.$$
\comment{Furthermore, if for any $w\in I$, there are finite number paths connecting $w$ to a representative set of $\{[v]\in \max(C(I))\}$ then we have 
            $$I=\bigoplus_{[v]\in \max(C(I))}\langle v \rangle.$$}

    \end{enumerate}
\end{thm}
\begin{proof}
    \begin{enumerate}[\upshape(1)]
        \item
            On one hand, observe that if $v\geq w$ then $\langle w \rangle \subseteq \langle v \rangle $ (since we are dealing with $\mathbb{Z}$-invariant order ideals). Conversely, a straightforward application of the Confluence Lemma~\ref{confuu}  and the description of $\langle v\rangle$ as given in (\ref{idealorderandnotorder})   yield the opposite implication. 
        \item
            Let $v\in P(E)$, and suppose that $\langle v \rangle =J\oplus K$, where $J$ and $K$ are $\mathbb{Z}$-order ideals of $T_E$. We will prove that $v\in J$ and thus $K=0$, or vice-versa. If $v \not \in J\cup K$, then  we may write $v=j+k$ for certain elements $j\in J$ and $k\in K$. By the Confluence Lemma, we may find $x,y$ in $F_{\overline{E}}$ such that $v\to x+y$, $j\to x$ and $k\to y$. Now, as $v\in P(E)$, there are two possibilities to consider.
        
            First, if $v$ is a sink, the only possibility is that $x=v$ and $y=0$ or vice-versa, in which case $v\in J$ or $v\in K$, as we wanted.
        
            Otherwise, $v$ lies in a cycle $\mathfrak{c}$ of $E$, so at least one among $x$ or $y$ (as an element of $F_{\overline{E}}$) will have a shift of a vertex $v'$ in $\mathfrak{c}$ as one of its summands. By symmetry of the argument, we may assume that this is so for $x$. Thus, $x\geq{}^nv'$ for some $n$. Since we are dealing with $\mathbb{Z}$-order ideals, we obtain that $v'\in J$, and so $\langle v' \rangle \subseteq J$. But $v\sim v'$, as both of these vertices lie in the cycle $\mathfrak{c}$, and so item (\ref{it:independencepathconnectednessideals}) yields $v\in \langle v \rangle =\langle v' \rangle \subseteq J$, again as we wanted.
            
            Therefore, $\langle v\rangle $ is indecomposable for any $v\in P(E)$. 
      
        \item
            First note that the equivalence of any element of $[v]$ being a sink, or belonging to a cycle without exits, or to an extreme cycle to the respective property of $v$, follows from $P_s(E)$, or $P_{cc}(E)$ or $P_{ec}(E)$ being $\sim$-invariant.
    
            Suppose that $v$ is a sink or that it belongs to a cycle without exits or to an extreme cycle. Let $J$ be a non-trivial $\mathbb{Z}$-order ideal of $\langle v \rangle$. Recall from (\ref{idealorderandnotorder}) that $\langle v \rangle $ consists of all $x\in T_E$ such that $x\leq\sum_{i=1}^k {}^{n_i}v$, for some   $n_i\in\mathbb{Z}$. Consider any nonzero element $x$ of $J$, so that $x\leq\sum_{i=1}^k {}^{n_i}v$ for certain $n_i$. By the Confluence Lemma~\ref{confuu}, we find $y,z\in F_{\overline{E}}$ such that $\sum_{i=1}^k {}^{n_i}v\to y+z$ and $x\to y$. As a term of $F_{\overline{E}}$, $y$ is a sum of shifts of vertices of $E$, and it is nontrivial as $x=y$ in $T_E$. Take any of these terms ${}^m w$. Then $x\geq {}^m w$ in $T_E$, so $w\in J$ as $J$ is a $\mathbb{Z}$-order ideal. Moreover, since $\sum_{i=1}^k {}^{n_i}v\to y+z$ and ${}^m w$ is a constituent of $y$, we obtain from the the definition of $\to$ that there is a path from $v$ to $w$. However, the given conditions on $v$ imply that there is a path from $w$ to $v$ as well:
            \begin{itemize}
                \item If $v$ is a sink then the only possibility is that $w=v$.
                \item If $v$ belongs to a cycle without exits, then the only possibility is that $w$ belongs to that same cycle, so just moving around this cycle gives us a path from $w$ to $v$.
                \item If $v$ belongs to an extreme cycle, then either $w$ already belongs to that cycle and we proceed as in the previous case, or it belongs to a path that exits this cycle,  and thus there is a return to $v$ by the extreme property.
            \end{itemize}
            Therefore, $w\sim v$ and hence $\langle v\rangle =\langle w\rangle $ by item~(\ref{it:independencepathconnectednessideals}). Since $w\in J$, we conclude that $\langle v \rangle \subseteq J$. Thus, $\langle v\rangle $ contains no nontrivial proper ideals.
    
            Conversely, suppose that $v$ does not have any of the properties given in the statement of this item. This means that $v$ belongs to a cycle $\mathfrak{c}$ which has an exit $\alpha$ with no return. Extending $\alpha$ as necessary (using the finiteness of $E^0$), we may assume that $w:=r(\alpha)$ is a sink or that it belongs to a cycle. Thus, $\langle w \rangle $ is a nontrivial ideal in $\langle v \rangle $. We prove that it is also proper, that is, that $v\not\in \langle w \rangle $.
            
            Indeed, the same type of argument as done in the other direction would tell us that, if $v$ were an element of $\langle w \rangle $ then there would be $x,y\in F_{\overline{E}}$ and integers $n_i$ such that $v\to x$ and $\sum_{i=1}^k {}^{n_i}w \to x+y$. However, as $v$ belongs to the cycle $\mathfrak{c}$, the definition of $\to$ implies that, as element of $F_{\overline{E}}$, $x$ has a term of the form ${}^m z$ for some vertex $z$ in $\mathfrak{c}$. Again by the definition of $\to$, the fact that $\sum_{i=1}^k {}^{n_i} w \to x+y$ implies that there is a path from $w$ to $z$, which would be a return from $w$ to $\mathfrak{c}$, which cannot happen.
    
            Therefore, if $v$ does not have any of the given properties then $\langle v \rangle$ is not simple. This completes the proof. 
        \item
            We first check that
            \begin{equation}\label{balconwid}
                 I \subseteq \sum_{[v]\in \max(C(I))}\langle v \rangle.
            \end{equation}
            It is enough to show that any vertex $x\in I$ is contained in the right hand side of (\ref{balconwid}). By Lemma~\ref{pesto}, 
            $x = \sum_{i=1}^n x_i(n_i)$, where 
            $x_i\in I$ are either sinks or are on a cycle. Thus $[x_i]\in C(I)$, for all $1\leq i \leq n$.  If $[x_i]$ is maximal, then $\langle x_i \rangle$ appears in the right of (\ref{balconwid}). If $[x_i]$ is not maximal, then there is a maximal element $[z] \in C(I)$ which appears in the right of (\ref{balconwid}) and    $[x_i] < [z]$. Thus $\langle x_i\rangle \subseteq \langle z \rangle$. This shows that all $x_i$ belong to the right hand side of (\ref{balconwid}) and so does $x$. Therefore (\ref{balconwid}) follows. The converse inclusion is trivial and thus $$
                 I = \sum_{[v]\in \max(C(I))}\langle v \rangle.\qedhere$$

     \end{enumerate}
\end{proof}

We finish the paper by unifying the decomposibility of graph algebras with their talented monoids.  Recall  that for a hereditary subset $H\subseteq E^0$, we say that a path $\alpha=e_1e_2\cdots e_n$ \emph{flows} to $H$ if 
$s(e_n)\not \in H$ but $r(e_n)\in H$. 

\begin{thm}\label{decomtal}
    Let \( E \) be a row-finite graph. The talented monoid \( T_E \) is decomposable if and only if there exist nontrivial hereditary and saturated subsets \( H \) and \( K \) such that \( H \cap K = \varnothing \) and, for any \( v \in E^0 \setminus (H \cup K) \), there is a nonzero but finite number of paths from \( v \) that flow to either \( H \) or \( K \).

\end{thm}
\begin{proof}

Suppose \( T_E = I \oplus J \), where \( I \) and \( J \) are \(\mathbb{Z}\)-order ideals of \( T_E \). Then \( I = \langle H \rangle \) and \( J = \langle K \rangle \), where \( H \) and \( K \) are hereditary and saturated subsets of \( E \) satisfying \( H \cap K = \varnothing \) and \( E^0 = \overline{H \cup K} \). By Lemma~\ref{gfhfhfh22},  choosing \( H \cup K \) as the hereditary set, it follows that for any \( v \in E^0 \setminus (H \cup K) \), there is a nonzero but finite number of paths from \( v \) that flow to either \( H \) or \( K \).

Conversely, suppose that $H\cap K=\varnothing$. This implies that  the $\mathbb Z$-order ideals $I=\langle H \rangle$  and $J=\langle K \rangle$ intersect trivially. Again choosing the hereditary subset $H\cup K$, an application of Lemma~\ref{gfhfhfh22} implies that $E^0=\overline {H \cup K}$. Thus, we conclude that $T_E=I\oplus J.$
\end{proof}

We can then relate the decomposition of graph algebras to that of their talented monoids.

\begin{cor}\label{fgfyr67433}
    Let $E$ be a row-finite graph. The following are equivalent:
\begin{enumerate}[\upshape(1)]    
    
  \item   the talented monoid $T_E$ is decomposable;
  
  \item the Leavitt path algebras $L_\K(E)$ is decomposable;
  
  \item the graph $C^*$-algebra $C^*(E)$ is decomposable. 
\end{enumerate}

\end{cor}
\begin{proof} The decomposability of the graph algebras \( L_\K(E) \) and \( C^*(E) \) was established in \cite{gonzalo} and \cite{hong}, respectively, under the exact same conditions as in Theorem~\ref{decomtal}.
    \end{proof}

{\bf Acknowledgment.}
D. Gon\c{c}alves was partially supported by Conselho Nacional de Desenvolvimento Cient\'ifico e Tecnol\'ogico (CNPq) and Funda\c{c}\~ao de Amparo \`a Pesquisa e Inova\c{c}\~ao do Estado de Santa Catarina (Fapesc) - Brazil. R. Hazrat acknowledges Australian Research Council Discovery Project Grant DP230103184. This project initiated at Mathematisches Forschungsinstitut Oberwolfach when the authors were visiting under Research in Pairs in 2022.


\begin{thebibliography}{99}

\bibitem{lpabook} G. Abrams, P. Ara, M. Siles Molina, Leavitt path algebras. \emph{Lecture Notes in Mathematics. Springer} (2017). 21--91.




\bibitem{Pardo} P. Ara, M. Moreno, E. Pardo, \textit{Nonstable K-theory for Graph Algebras}, Algebr. Represent. Theory 10 (2007), 157--178.


\bibitem{gonzalo} G. Aranda Pino, A. Nasr-Isfahani, \emph{Decomposable Leavitt path algebras for arbitrary graphs,}
Forum Math. 27 (2015), no. 6, 3509–-3532.
 
 \bibitem{galera} J.H. Brown, A.H. Fuller, D. R. Pitts, S. A. Reznikoff, \textit{Regular ideals of graph algebras}, Rocky Mountain Journal of Mathematics 52 (1), 43--48.
 
 \bibitem{CCL} C.G. Canto, D. Mart\'{\i}n Barquero, C. Mart\'{\i}n Gonz\'alez, \textit{Invariant ideals in Leavitt path algebras}, Publ. Mat. 66 (2022), 541--569.
 
 
 


\bibitem{Luiz}  L.G. Cordeiro, D. Gon\c{c}alves, R. Hazrat, \emph{The talented monoid of a directed graph with applications to graph algebras}, 
Revista Matem\'atica Iberoamericana,  38 (2022), no. 1, 223--256.

\bibitem{meteorgraphs} L.G. Cordeiro, E. Gillaspy, D. Gon\c{c}alves, R. Hazrat, \emph{Williams' Conjecture holds for meteor graphs}. To appear in Israel Journal of Mathematics.


\bibitem{willie}  
G. Corti\~nas, R. Hazrat, \emph{Classification conjectures for Leavitt path algebras}, Bulletin  of London Math Society,  56 (2024), no. 10, 3011--3060.


\bibitem{dd0} D. Gon\c{c}alves, D. Royer, \textit{A note on the regular ideals of  Leavitt path algebras},   Journal of Algebra and its Applications, vol 21, no 11, 2250225, 2022.





  

\bibitem{hazli} R. Hazrat, H. Li, \emph{The talented monoid of a Leavitt path algebra}, Journal of Algebra, 547 (2020) 430--455.

\bibitem{liahaz} R. Hazrat, L. Va\v s, \emph{Comparability in the graph monoid}, 
New York J. of Math, 26 (2020), 1375--1421. 

\bibitem{hong} J.H. Hong, \emph{Decomposability of graph $C^*$-algebras}, Proc. Amer. Math. Soc. 133 (1) (2004), 115–-126.

\bibitem{lam} T.Y. Lam, Lectures on modules and rings, Graduate Texts in Mathematics, 189, Springer, New York, 1999.

\bibitem{raeburn} I. Raeburn, Graph algebras. CBMS Regional Conference Series in Mathematics, 103
American Mathematical Society, Providence, RI, 2005.





\bibitem{Vas1} L. Va\v s, \textit{Annihilator ideals of graph algebras}, Journal of Algebraic Combinatorics, 58 (2) (2023), 331--353.

\bibitem{Vas2} L. Va\v s, \textit{Graph characterization of the annihilator ideals of Leavitt path algebras,} Bulletin of the Australian Mathematical Society, 110 (2024), 498--507.

\bibitem{Vas3} L. Va\v s, \textit{
Every graded ideal of a {L}eavitt path algebra is graded isomorphic to a {L}eavitt path algebra}, 	Publicacions Matematiques, 69 (2025), 45--82.









 










\end{thebibliography}
\end{document}